\documentclass{article}
\usepackage[utf8]{inputenc}

\usepackage{amsmath}
\usepackage{amsthm}
\usepackage{amssymb}
\usepackage{upgreek}
\usepackage{xcolor}
\usepackage{comment}
\usepackage{tikz-cd}

\theoremstyle{plain}
\newtheorem{thm}{Theorem}[section]
\newtheorem{lemma}[thm]{Lemma}
\newtheorem{prop}[thm]{Proposition}
\newtheorem{cor}[thm]{Corollary}
\theoremstyle{definition}
\newtheorem{defn}[thm]{Definition}

\def\R{\mathbb{R}}

\def\F{\mathbb{F}}
\def\Z{\mathbb{Z}}
\def\N{\mathbb{N}}

\def\c{\cdot}

\title{Classification of braces of cardinality $p^4$}
\author{Dora  Pulji\'{c}\thanks{School of Mathematics, JCMB, The King's Buildings, University of Edinburgh, EH9 3BF, d.puljic@sms.ed.ac.uk} }

\begin{document}

\maketitle

\begin{abstract}

 We classify nilpotent pre-Lie rings of cardinality $p^4$ and thereby braces of the same cardinality, for a sufficiently large prime $p$. It has been shown that nilpotent pre-Lie rings of cardinality $p^n$ correspond to strongly nilpotent braces of the same cardinality, for sufficiently large $p$. These braces are explicitly obtained from the corresponding pre-Lie rings by the construction of the group of flows. Not right nilpotent braces of cardinality $p^4$ have been classified, hence our results finish the classification of braces of cardinality $p^4$.

\end{abstract}


\tableofcontents
\section{Introduction}
 The connection between pre-Lie algebras over $\R$ and left nilpotent $\R$-braces was described in \cite{MR3291816}, along with a method for constructing a brace from a pre-Lie algebra. This connection was further investigated in \cite{MR4353236}, where a formula for the correspondence between strongly nilpotent $\F_p$-braces and nilpotent pre-Lie algebras over $\F_p$ was given. Further, in \cite{passage}, it was shown that nilpotent pre-Lie rings of cardinality $p^n$ correspond to strongly nilpotent braces of the same cardinality, for $p>n+1$. These braces have the same additive group as the pre-Lie algebra and can be explicitly calculated from the pre-Lie algebra by constructing the group of flows. Every brace $A$ (not necessarily right nilpotent) gives rise to a pre-Lie ring related to the factor brace $A/ann(p^2)$ \cite{shalev2022braces}, and this construction is reversible \cite{smoktunowicz2022prelie}. It is not clear whether every brace of cardinality $p^4$ corresponds to a pre-Lie ring. 

It follows that by characterising nilpotent pre-Lie rings of cardinality $p^n$ we achieve the classification of strongly nilpotent braces of cardinality $p^n$, for a sufficiently large prime $p$. Cyclic braces were classified in \cite{rump2007classification, rump2019classification}, braces of cardinality $pq$ have been classified in \cite{acri2020skew} and of cardinality $p^2q$ in \cite{dietzel2021braces}, braces of cardinality $p^3$ have been classified in \cite{MR3320237}, skew braces of cardinality $p^3$ have been described in \cite{nejabati2018hopf}, and all not right nilpotent $\mathbb F_{p}$-braces of cardinality $p^{4}$ were described in \cite{puljic2021braces}. It was shown in \cite{puljic} that all braces of cardinality $p^4$, apart from those constructed in \cite{puljic2021braces}, are right nilpotent. It follows that the remaining braces of cardinality $p^4$ to be classified are right nilpotent, for which there is a passage to nilpotent pre-Lie rings. Nilpotent pre-Lie algebras of cardinality $p^4$ over the field $\F_p$ generated by $1$ element have been classified in \cite{MR4353236}, and thereby strongly nilpotent $\F_p$-braces generated by $1$ element of the same cardinality. In this paper we classify the remaining nilpotent pre-Lie rings of cardinality $p^4$, and hence we finish the classification of braces of cardinality $p^4$.

\section{Preliminaries}

\begin{defn}
    A  \textbf{brace} is a triple $(A, +,\circ)$ where $(A,+)$ is an abelian group, $(A,\circ)$ is a group and 
\[a\circ(b+c)+a=a\circ b+a\circ c\]
for all $a,b,c\in A$.
\end{defn}
 The star operation  $\ast$ is  defined  as \begin{equation*}
\label{star}
    a \circ b = a \ast b + a + b.
\end{equation*}
Then, equivalently, a brace is a triple $(A, +,\ast)$ where $(A,+)$ is an abelian group, $(A,\ast)$ is a group and 
\[a\ast(b+c)=a\ast b+a\ast c\]
for all $a,b,c\in A$. We will refer to $(A,\circ)$ as the multiplicative group of the brace.

A brace $A$ is \textbf{left nilpotent} if there exists $n\in\N$ such that $A^{n}=0$, where $A^{i+1}=A\ast A^i$ and $A^1=A$. A brace is \textbf{right nilpotent} if there exists $n\in\N$ such that $A^{(n)}=0$, where $A^{(i+1)}=A^{(i)}\ast A$ and $A^{(1)}=A$. A brace is \textbf{strongly nilpotent} if there exists $n\in\N$ such that $A^{[n]}=0$, where $A^{[i+1]}=\sum_{j=1}^{i}A^{[j]}*A^{[i+1-j]}$ and $A^{[1]}=A$. The smallest such $n$ is called the nilpotency index of $A$.

\begin{defn}
   A \textbf{pre-Lie algebra} is a triple $(A,+,\c)$ consisting of a $k$-vector space $A$ with a binary operation $(x,y)\mapsto x\c y$ such that 
\[(a\c b)\c c -a\c(b\c c) = (b\c a)\c c - b\c(a\c c),\]
\[(ia+jb)\c c = i(a\c c) +j (b\c c)\quad \text{ and }\quad a\c (ib+jc) = i(a\c b)+j(a\c c)\]
for all $a,b,c\in A$ and $i,j\in k$.

\end{defn}

\begin{defn}
   A \textbf{pre-Lie ring} is a triple $(A,+,\c)$ consisting of a $\Z$-module $A$  with a binary operation $(x,y)\mapsto x\c y$ such that 
\[(a\c b)\c c -a\c(b\c c) = (b\c a)\c c - b\c(a\c c),\]
\[a\c(b+c)=a\c b+a\c c \quad \text{ and }\quad (a+b)\c c = a\c c+ b\c c\]
for all $a,b,c\in A$.
\end{defn}

We say that a pre-Lie ring is \textbf{nilpotent} (or strongly nilpotent) if for some $n\in\N$ any product of $n$ elements is zero. We denote the radical chains in pre-Lie rings the same way as in braces, but using the binary operation of the pre-Lie ring.

	\begin{defn}
	 	Let $ \mathbb{F} $ be a field. We say that a left brace $ A $ is an $ \mathbb{F} $-\textbf{brace} if its additive group is an $ \mathbb{F} $-vector space such that \[ a \ast(\alpha b) = \alpha (a\ast b) \] for all $ a, b \in A $  and $ \alpha \in \mathbb{F} $. 
	\end{defn}

Note that an $\F_p$-brace of cardinality $p^n$ is a brace with additive group $C_p^n$.

\subsection{From pre-Lie rings to braces}\label{method}

In this section we give justification for claiming that our characterisation of pre-Lie rings gives a complete description of braces of cardinality $p^4$.

Again, note that in \cite{puljic} it was shown that all braces of cardinality $p^{4}$ are right nilpotent, except for  braces which were constructed in \cite{puljic2021braces}. Therefore it remains to describe all braces of cardinality $p^{4}$ which are strongly nilpotent. The classification of strongly nilpotent braces of cardinality $p^4$, for a prime $p$ such that $p>5^5$, is achieved through the classification of nilpotent pre-Lie rings of cardinality $p^4$. The reasoning is as follows.
\begin{enumerate}
    \item Braces of cardinality $p^4$ are left nilpotent \cite{rump2007braces}. Left nilpotent braces that are right nilpotent are strongly nilpotent \cite{MR3814340}.
    \item By Corollary $19$ from \cite{passage}, if $A$ is a  strongly nilpotent brace of cardinality $p^{n}$ for some $n$, then  the nilpotency index $k$ of $A$ satisfies  $k<(n+1)^{n+1}$. Therefore all strongly nilpotent braces of cardinality $p^{4}$ have nilpotency index $k<5^{5}$.
    \item If $p>5^{5}$, then  all strongly nilpotent braces of cardinality $p^{4}$ have nilpotency index $k<5^{5}<p-1$. Such braces satisfy assumptions from Theorem $6$ in \cite{passage}.
    \item  Further, by Theorem $6$ from  \cite{passage} the formula \[a\cdot b=-(1+p+\ldots +p^{p})\sum_{i=0}^{p-1}\xi ^{p-1-i}((\xi ^{i}a)*b)\] gives a pre-Lie ring $(A, +, \cdot )$ from a brace $(A,+,\ast)$.
    \item  The brace $(A, +, \ast )$ can be recovered from the obtained pre-Lie ring  by applying the construction of the group of flows to the pre-Lie ring $(A, +, \cdot)$ (see the main result from  \cite{passage}). 
    \item Therefore, for $p>5^{5}$,  every strongly nilpotent brace of cardinality $p^{4}$, is obtained by applying the construction of group of flows to some powerful pre-Lie ring (which is also strongly nilpotent). 

\end{enumerate}

Now, given a pre-Lie ring we construct the corresponding brace as follows (described in \cite{MR3291816}).

\begin{enumerate}

\item Let $(A, +, \cdot )$ be a pre-Lie ring of cardinality $p^{n}$ for some prime number $p$ and some $n\in \N$. By Lemma $11$ from \cite{passage}, every element of $A$ can be written as $W(x)$ for some $x\in A$. 
Recall that \[W(x)=x+{\frac 12} x\cdot x+{\frac 1{3!}}x\cdot (x\cdot x)+ {\frac 1{4!}}x\cdot (x\cdot (x\cdot x))+\cdots .\] 
For $x,y\in A$ we define 
 \[W(x)*W(y)=x\cdot W(y)+{\frac 1{2!}}x\cdot (x\cdot W(y))+ {\frac 1{3!}}x\cdot (x\cdot ( x\cdot  W(y)))+\cdots .\]
 
\item We further define \[W(x)\circ W(y)=W(x)*W(y)+W(x)+W(y).\] 
Then $(A, +, \circ )$ is the brace corresponding to the pre-Lie ring $(A, +, \cdot)$.

\end{enumerate}

This construction can be also described using the construction of the group of flows from \cite{MR579930}. See \cite{passage} for a detailed description.

Note that in the case where a pre-Lie algebra $(A, +, \cdot )$ is such that $A^{[4]}=0$, the construction of the related brace is simpler. Following Example on page 4 of \cite{MR4391819}, for a pre-Lie algebra $(A, +, \cdot )$ with $A^{[4]}=0$ and $a,b\in A$ we define
\[a\circ b=a+b+a\cdot b-{\frac 12}(a\cdot a)\cdot b+{\frac 12} a \cdot (a\cdot b).\]
  Then $(A, +, \circ )$ gives a brace which corresponds to this pre-Lie algebra. This formula gives the same brace as described in \cite{MR3291816}.

\section{Classification}
As aforementioned, the current next step in the classification of braces is that of braces of cardinality $p^4$. Braces with additive group $C_{p^4}$, also called cyclic braces, were classified in \cite{rump2007classification, rump2019classification}. Braces with additive group $C_p^4$ that are not right nilpotent were classified in \cite{puljic2021braces}. Right nilpotent braces with additive group $C_p^4$ that are generated by $1$ element were described in \cite{MR4353236}. The remaining braces, with additive groups $C_p\times C_{p^3}$ and $C_{p^2}\times C_{p^2}$, were shown to be right nilpotent in \cite{puljic}. The characterisation of the remaining braces - right nilpotent braces with additive group $C_p^4$ that are generated by more than $1$ element and braces with additive groups $C_p\times C_{p^3}$ and $C_{p^2}\times C_{p^2}$, is obtained through the description of the corresponding pre-Lie rings or pre-Lie algebras and application of section \ref{method}.

For the remainder of this paper we assume $p$ is a prime number with $p>5^5$.

\subsection{Nilpotent pre-Lie algebras with additive group $C_p^4$}\label{cp4}
 
One generator nilpotent  pre-Lie algebras with additive group $C_p^4$  have been classified  in \cite{MR4353236}, and thereby one generator strongly nilpotent $\F_p$-braces of cardinality $p^4$. We continue the classification of nilpotent pre-Lie algebras with additive group $C_p^4$ with more than one generator using similar methods. Result \cite[Lemma 10]{MR4353236} is used in the reasoning of the propositions below. We denote the additive group of the pre-Lie algebra $A$ by $(A,+)$.

\subsubsection{Pre-Lie algebras $A$ with $A^{[2]}\neq 0$ and $A^{[3]}=0$}

\begin{prop}\label{gens}
    Let $A$ be a pre-Lie algebra with $(A,+)\cong C_p^4$ that is generated by $2$ elements $x$ and $y$ as a pre-Lie algebra. If $A^{[2]}\neq 0$ and $A^{[3]}=0$, then we have the following:
    \begin{enumerate}
        \item $A$ is an $\F_p$-vector space generated by $\{x,y,z,w\}$, where $z,w\in A^{[2]}$ and $z\not\in \F_p(w)$.
        \item For $i,j\in \{x,y\}$ we have
        \[i\c j = \alpha_{ij}z+\beta_{ij}w,\]
        where at least two pairs $(\alpha_{ij}, \beta_{ij})$ do not equal $(0,0)$, and are not multiples of each other.
    \end{enumerate}
\end{prop}

\begin{proof}
    Suppose $A$ is an $\F_p$-vector space generated by $\{x,y,z,w\}$, for some $z,w\in A$. Notice that $x\not\in A^{[2]}$, as in that case either $x$ is not a generator of $A$ or $x=0.$ Similarly, $y\not\in A^{[2]}$. Now, suppose $z\not\in A^{[2]}$, so $z = a_1 x+a_2 y+a_3 x^2 +\dots$ for some $a_i\in \F_p$. Then either $z\in A^{[2]}$ or $z$ is linearly dependent on $x$ and $y$, and hence not a basis element. It follows that $z,w\in A^{[2]}$, and $z$ and $w$ are not linearly dependent. It follows that least two pairs $(\alpha_{ij}, \beta_{ij})$ do not equal $(0,0)$.
\end{proof}

\begin{prop}
Let $A$ be a pre-Lie algebra with $(A,+)\cong C_p^4$ that is generated by $3$ elements $x, y$ and  $z$ as a pre-Lie algebra. Then we have the following:
\begin{enumerate}
    \item $A^{[2]}\neq 0$ and $A^{[3]}=0$.
    \item $A$ is generated by $\{x,y,z,w\}$ as a vector space with $w\in A^{[2]}$.
    \item $A^{[2]}={\mathbb F}_p w$.
    \item For $i,j\in \{x,y\}$ we have
        \[i\c j = \alpha_{ij}w,\]
        where at least one $\alpha_{ij}\neq 0$.
    
\end{enumerate}
 
\end{prop}

\begin{proof}
Suppose $A^{[2]}=0$. Then the dimension of $A$ as a vector space over $\F_p$ would be at most $3$, which is not enough. Hence $A^{[2]}\neq 0$.

Notice that $x,y,z\not\in A^{[2]}$ (as well as their combinations) by a similar argument as in Proposition \ref{2el}. Hence $A/A^{[2]}$ is at least $3$-dimensional. Now, $A^{[2]}\neq 0$ implies $A\neq A^{[2]}$, so $A^{[2]}$ is $1$-dimensional. Hence $A^{[3]}=0$.

There exists at least one non-zero element of $A^{[2]}$ which we denote by $w$. Then $w$ is in the basis of $A$ as a vector space. If there exists another element $b\in A^{[2]}$ which is not in ${\mathbb F}_p w$, then $A$ is of dimension greater than $4$ as a vector space.
\end{proof}

\subsubsection{Pre-Lie algebras $A$ with $A^{[3]}\neq 0$ and $A^{[4]}=0$}

\begin{lemma}

Let $A$ be a pre-Lie algebra with $(A,+)\cong C_p^4$ that is generated by $2$ elements $x$ and  $y$ as a pre-Lie algebra. Then $A^{[4]}=0$.

\end{lemma}

\begin{proof}

Note that if $A^{[i]}=A^{[i+1]}$ for some $i\in \{1,2,3\}$, then $A^{[4]}=0$ by argument as in Lemma 10 of \cite{MR4353236}. 

Suppose $A\neq A^{[2]}\neq A^{[3]}\neq A^{[4]}$. This implies that $A^{[i]}/A^{[i+1]}$ has dimension $1$ or $2$. As $A$ is $4$-dimensional as a vector space over $\F_p$, $A^{[2]}$ can be $3$- or $2$-dimensional, $A^{[3]}$ can be $2$- or $1$-dimensional and $A^{[4]}$ can be $1$-dimensional or $0$. If $A^{[4]}$ is not $0$, then it is $1$-dimensional. Further $A^{[3]}$ is $2$-dimensional and $A^{[2]}$ is $3$-dimensional.
Now notice that by a similar argument as in Proposition \ref{gens} $x,y\not\in A^{[2]}$. Hence $A/A^{[2]}$ is at least $2$-dimensional, which is a contradiction.
Hence, $A^{[4]}=0$. 
\end{proof}

\begin{prop}\label{2el} 
Let $A$ be a pre-Lie algebra with $(A,+)\cong C_p^4$ that is generated by $2$ elements $x$ and $y$ as a pre-Lie algebra. If $A^{[3]}\neq 0$ and $A^{[4]}=0$, then we have the following:
 \begin{enumerate}
    \item Not all of $S:=\{x\c y, y\c x, x^2, y^2\}$ is in $A^{[3]}$.\label{firstp}
     \item $A$ is an $\F_p$-vector space generated by $\{x,y,z,w\}$ with $z\in S-A^{[3]}$ and $w$ has one factor in $\{x,y\}$ and the other in $S-A^{[3]}$. \label{basis}
     \item $A^{[3]}={\mathbb F}_p w$, $A^{[2]}/A^{[3]}={\mathbb F}_p z.$
     \item $A^{[2]}$ is $2$-dimensional as a vector space.
     \item For $i,j\in\{x,y\}$ we can write
     \[i\c j = \alpha_{ij}z+\beta_{ij}w\]
     for some $\alpha_{ij}, \beta_{ij}\in \F_p$. Then not all $\alpha_{ij}$ are zero. Further, for $k,l$ such that either $z=k$ and $l\in\{x,y\}$  or $z=l$ and $k\in \{x,y\}$ we let
     \[k\c l = \gamma_{kl}w.\]
     Then it follows that
     \[\alpha_{xy}\gamma_{zx}-\alpha_{yx}\gamma_{xz}-\alpha_{yx}\gamma_{zx}+\alpha_{xx}\gamma_{yz}=0,\]
     \[\alpha_{yx}\gamma_{zy}-\alpha_{xy}\gamma_{yz}-\alpha_{xy}\gamma_{zy}+\alpha_{yy}\gamma_{xz}=0.\]\label{relations}
     \item If $A$ satisfies properties \ref{firstp} to \ref{relations}, then it is a well defined pre-Lie algebra.
 \end{enumerate}

\end{prop}

\begin{proof}
Let $S:=\{x\c y, y\c x, x^2, y^2\}$. Notice that by a similar argument as in Proposition \ref{gens} $x,y\not\in A^{[2]}$. Furthermore, not all of $S$ can be products of more than two elements as this would lead to $A^{[2]}=0$. Hence at least one of $x\c y, y\c x, x^2, y^2$ is (non-zero and) not a product of $3$ elements. Let this element be $z$. As $x,y\not\in A^{[2]}$ we have that $A/A^{[2]}$ is $2$ dimensional, $A^{[2]}/A^{[3]}$ is $1$ dimensional and $A^{[3]}$ is $1$ dimensional. Therefore, elements of $A^{[2]}$ other than $z$ are a product of $3$ elements or in ${\mathbb F}_pz$. 

Since $A^{[3]}$ is $1$ dimensional there has to exist a non-zero product of $3$ elements (and of no more than $3$ elements). Let this be $w$. Notice that $w$ is a product $c\cdot d$ with $d\in \{x,y\}$ and $c\in S$ (or vice versa). Then $c\in S-A^{[3]}$ as otherwise $z=0$. Any non-zero element of $A^{[3]}$ is in ${\mathbb F}_p w$ as $A^{[3]}$ is $1$ dimensional. 

Since $A^{[2]}$ is $2$-dimensional, we can write any element of $A^{[2]}$ as a linear combination of $z$ and $w$, as per point \ref{relations}. Notice that the only relevant pre-Lie algebra relations are
\begin{align}
        (x\c y)\c x-x\c (y\c x)&=(y\c x)\c x-y\cdot x^2,\\
        (y\c x)\c y-y\c (x\c y)&=(x\c y)\c y-x\cdot y^2
    \end{align}
as products of more than $3$ elements are zero. Using the notation specified in the proposition statement and the pre-Lie algebra relations, we arrive at the relations in \ref{relations}. Notice that this implies that the structure is indeed a well-defined pre-Lie algebra.

\end{proof}

\subsection{Nilpotent pre-Lie rings with additive group $ C_p \times C_{p^3} $}\label{cpcp3}

Braces with additive group $C_p \times C_{p^3}$ are right nilpotent, hence we can calculate them from nilpotent pre-Lie rings with the same additive group. In this section we characterise nilpotent pre-Lie Rings $(A,+,\c)$ with $(A,+)\cong  C_p \times C_{p^3}$. Note that by analogous arguments as in \cite[Lemma 10]{MR4353236} we have $A^{[6]}=0$.

\subsubsection{Pre-Lie rings $A$ with $A^{[2]}\neq 0$ and $A^{[3]}=0$}

\begin{prop}\label{a3=0cp3}
    Suppose $A$ is a pre-Lie ring with $(A,+)\cong C_p\times C_{p^3}$ generated by $x$ and $y$ as an additive group, with $py=p^3x=0$. Further, suppose $A$ is generated by two elements as a pre-Lie ring. If $A^{[2]}\neq 0$ and $A^{[3]}=0$, then we have the following:
    \begin{enumerate}
        \item $A^{[2]}$ is generated by $p^2 x$.
        \item We have
        \[x\cdot x = ax, \;\; x\cdot y = cx,\;\; y\cdot x = ex, \;\; y\cdot y = gx,\]
        for $a,c,e,g\in \{1,2, \dots, p^3\}$ such that $p^2\mid a,c,e,g.$
    \end{enumerate}
    
\end{prop}

\begin{proof}
    Note that we can assume that $A$ is generated by $x$ and $y$ as a pre-Lie ring. Following a similar argument as in Proposition \ref{gens}, $x,y\not\in A^{[2]}$. We let for $a,c,e,g\in \{1,2, \dots, p^3\}$ and $b,d,f,h\in\{1,2,\dots,p\}$
\[x\cdot x = ax+by, \;\; x\cdot y = cx+dy,\;\; y\cdot x = ex+fy, \;\; y\cdot y = gx+hy.\]
Then $p(x\c y) = p(y\c x) = p(y\c y)=0$ implies that 
$p^2\mid c,e,g$. We have $p (x\cdot x) = pax$, hence 
\[px\c(x\c(x\c(x\c (x\c x))))= pa^5 x \in A^{[6]}.\]
Hence $p\mid a$. It follows that  $x\c(x\c x) = a^2x=0$, hence $p^2\mid a$. This implies that $A^{[2]}$ is generated by $p^2x$.
\end{proof}

\begin{prop}
    Suppose $A$ is a pre-Lie ring with $(A,+)\cong C_p\times C_{p^3}$ generated by $x$ and $y$ as an additive group, with $py=p^3x=0$. Further, suppose $A$ is generated by one element as a pre-Lie ring. If $A^{[2]}\neq 0$ and $A^{[3]}=0$, then the following holds:
    \begin{enumerate}
        \item We have $y = \alpha x + \beta x^2$, for $\alpha, \beta\in \{1,2, \dots, p^3\} $ with $p\mid \alpha$.
        \item $A^{[2]} = \Z (p^2x)$.

    \end{enumerate}
    
\end{prop}

\begin{proof}
    Notice that we can assume that $x$ is the generator of $A$ as a pre-Lie ring. Hence we can write $y=\alpha x+\beta x^2$ for some $\alpha, \beta \in \{1,2, \dots, p^3\}$. If $p\nmid \alpha$, then we can write
    \[px = -\alpha^{-1}\beta x\c px = \alpha^{-2}\beta^2 x\c (x\c x)=0. \]
    Therefore $p\mid \alpha$.
    
    Now suppose $A^{[2]}$ is generated by $px$ as an additive group. Then $|A^{[2]}| = p^2$. Notice that $px^2 = 0$ as it is in $A^{[3]}$. As $A^{[2]} = \Z(x^2)$, this implies $|A^{[2]}| = p$, which is a contradiction. Therefore, $A^{[2]}$
    is generated by $p^2x$ as an additive group.

\end{proof}

\subsubsection{Pre-Lie rings $A$ with $A^{[3]}\neq 0$}

 In Lemmas \ref{first} to \ref{b=0} we assume that $A$ is a nilpotent pre-Lie ring with $A^{[3]}\neq 0$, with additive group $(A,+)\cong C_p\times C_{p^3}$ generated by $x$ and $y$, where $p^3x=py=0.$ Further, we let for $a,c,e,g\in \{1,2, \dots, p^3\}$ and $b,d,f,h\in\{1,2,\dots,p\}$
\[x\cdot x = ax+by, \;\; x\cdot y = cx+dy,\;\; y\cdot x = ex+fy, \;\; y\cdot y = gx+hy.\]

\begin{lemma}\label{first}
$p^2\mid a,c,e,g$.
\end{lemma}

\begin{proof}
Follows from the same argument as in Proposition \ref{a3=0cp3}.
\end{proof}

\begin{lemma}
$p\mid h$.
\end{lemma}

\begin{proof}
We have $y\c (y\cdot y) =h y\c y$, hence 
\[y\c(y\c(y\c(y\c (y\c y))))=h^4 y\c y= h^4 gx + h^5 y\in A^{[6]}.\]
Hence $p\mid h$.
\end{proof}

\begin{lemma}
$p\mid d$.
\end{lemma}

\begin{proof}
We have $x\c (x\cdot y) =c x\c x+d x\c y =dx\c y$, hence 
\[x\c(x\c(x\c(x\c (x\c y))))=d^4 x\c y= cd^4 x+d^5 y\in A^{[6]}.\]
Hence $p\mid d$.
\end{proof}

\begin{lemma}
$p\mid f$.
\end{lemma}

\begin{proof}
We have $(y\c x)\cdot x =e x\c x+f y\c x = fy\c x$, hence 
\[((((y\c x)\c x)\c x)\c x)\c x) =f^4 y\c x = f^4e x+ f^5 y\in A^{[6]}.\]
Hence $p\mid f$.
\end{proof}

\begin{lemma}
    $A^{[4]}=0$.
\end{lemma}

\begin{proof}
    The only possible non-zero elements of the form $x_1\c (x_2\c x_3)$ and $(x_1\c x_2)\c x_3$ for $x_i\in \{x,y\}$ are 
    \[y\c(x\c x) = bgx = (x\c x)\c y,\quad  x\c(x\c x) = (a^2 +bc) x, \quad (x\c x)\c x = (a^2+be)x.\]

From the pre-Lie relations it follows that 
\[(y\c x)\c x - y\c (x\c x) = (x\c y)\c x - x\c(y\c x),\]
so $y\c (x\c x) = bg x = 0,$ so $p^3 \mid bg.$ We have
\begin{align*}
    (x\c x)\c(x\c x) &= a^2x\c x +b^2 y\c y\\
    &=b^2 gx\\
    &= 0.
\end{align*}
It can be easily checked that every other product of $4$ elements in $A$ is $0$.

\end{proof}

\begin{lemma}\label{b=0}
    $x,y\not\in A^{[2]}$, $p^2\nmid a$ and $b=0$.
\end{lemma}

\begin{proof}
    Suppose $x\in A^{[2]}$. Then $x\c x=0$ and $x\c y, y\c x\in A^{[3]}$, so $p^2x\in A^{[3]}$. It follows that $y\c y\in A^{[3]}$, implying that $A^{[2]}=A^{[3]}$, so $A^{[3]}=0$ by Lemma 10 of \cite{MR4353236}. Hence, $x\not\in A^{[2]}$.

    Suppose $y\in A^{[2]}$. Then, similarly, $p^2x\in A^{[3]}$ and $x\c x \not\in A^{[3]}$, as otherwise $A^{[3]}=0$. Now notice that the only non-zero products of three elements in $A$ are multiples of $x\c(x\c x) $ and $(x\c x)\c x$, that is, multiples of $p^2x$. Hence $y\not\in A^{[3]}$. Then $x\c x = ax+by\not\in A^{[3]}$ implies $px\not\in A^{[3]}$. It follows that $|A^{[2]}/A^{[3]}|\geq p^2$, contradicting the fact that the only elements of $A^{[2]}/A^{[3]}$ are  $\Z (x\c x)$.

    It follows that $x,y\not\in A^{[2]}$, so $b=0$, as otherwise $x\c x\not\in A^{[2]}$. Notice that $p^2x\in A^{[3]}$, so $p^2\nmid a$, as otherwise $x\c x \in A^{[3]}$.
\end{proof}

\begin{prop}
    Let $A$ be a pre-Lie ring with $(A,+)\cong C_p\times C_{p^3}$, generated by $x$ and $y$ as an abelian group with $p^3x=py=0$. Further, suppose we have
\[x\cdot x = ax, \;\; x\cdot y = cx,\;\; y\cdot x = ex, \;\; y\cdot y = gx\]
    for $a,c,e,g\in \{1, 2,\dots, p^3\}$ such that $p^2\mid c,e,g$, $p\mid a$ and $p^2\nmid a$. Then $A$ is a well-defined pre-Lie ring with $A^{[4]}=0$ of cardinality $p^4$. 
\end{prop}

\begin{proof}
    To check that pre-Lie relations hold in $A$ it suffices to check whether
    \[x\c(y\c x)-(x\c y)\c x = y\c (x\c x)- (y\c x)\c x\]
    and
    \[x\c(y\c y)-(x\c y)\c y = x\c (y\c y)- (x\c y)\c y\]
    
    hold, as we can write any element of $A$ as $\alpha x+\beta y$ for $\alpha\in \{1,2,\dots, p^3\}$, $\beta\in\{1,2,\dots, p\}$. Notice that the only non-zero products of $3$ elements of $A$ are multiples of $x\c(x\c x)$ and $(x\c x)\c x$, so the pre-Lie relations hold. We also have
    \[(x\c x)\c (x\c x) = a^3 x = 0,\]
    so it follows that $A^{[4]}=0.$

    Notice that $A$ has cardinality $p^4$ as $A^{[3]} = \{kp^2x\mid k\in\{1,2,\dots,p\}\}$, $A^{[2]}/A^{[3]} = \{kpx+A^{[3]}\mid k\in \{1,2,\dots,p\}\}$, and $A/A^{[2]} = \{kx+ly+A^{[2]}\mid k,l\in \{1,2,\dots,p\}\}$.
\end{proof}

\subsection{Nilpotent pre-Lie rings with additive group $C_{p^2}\times C_{p^2}$}\label{cp2cp2}

Braces with additive group $C_{p^2} \times C_{p^2}$ are right nilpotent, hence we can calculate them from nilpotent pre-Lie rings with the same additive group. In this section we characterise pre-Lie rings $(A,+,\c)$ with the additive group $(A,+)\cong C_{p^2} \times C_{p^2}$.

Note that by the same argument as in Lemma 10 of \cite{MR4353236} we have $A^{[6]}=0$. Further, Lemma 10 shows that if $A\neq A^{[2]}\neq A^{[3]}\neq A^{[4]}\neq 0$, then we have
 $|A^{[i]}/A^{[i+1]}|=p$ for $i=1,\dots, 3$, so $|A^{[2]}|=p^3$, $|A^{[3]}|=p^2$, and $|A^{[4]}|=p$. Moreover, if $A^{[i]}=A^{[i+1]}$ for some $i=1,2,3$, then $A^{[3]}=0$. 

\subsubsection{Pre-Lie rings $A$ with $A^{[2]} \neq 0$ and $A^{[3]} = 0$}

In this section we assume that $A^{[2]}\neq 0$ and $A^{[3]}=0$.

\begin{prop}
    Let $A$ be a pre-Lie ring with $A^{[2]}\neq 0$ and $A^{[3]}=0$. Suppose the additive group $(A,+)\cong C_{p^2} \times C_{p^2}$ is generated by $x$ and $y$, and that $A$ is generated by two elements as a pre-Lie ring. Then the following holds:
    \begin{enumerate}
        \item $A^{[2]}$ is generated by $px$ and $py$ as an additive group.
        \item  We have
        \[x\c x= ax+by,\quad  x\c y = cx+dy,\quad  y\c x = ex+fy,\quad  y\c y = gx+hy,\]
        where $a,b,c,d,e,f,g,h\in \{1,2,\dots,p^2\}$, not all zero, with $p\mid a,b,c,d,e,f,g,h$.
    \end{enumerate}
\end{prop}

\begin{proof}
    Notice that we can let $x$ and $y$ be generators of $A$ as a pre-Lie ring. Arguing similarly to Proposition \ref{gens} $x,y\not\in A^{[2]}$. Suppose $A^{[2]}$ is generated by $px$ as an additive group. We can let 
    \[x\c x= ax,\quad  x\c y = cx,\quad  y\c x = ex,\quad  y\c y = gx\]
    for $a,b,c,d,e,f,g,h\in \{1,2,\dots,p^2\}$ with $p\mid a,c,e,g$. Now notice that $(x\c x)\c x = a^2x = 0$ implying $p^2\mid a$. Similarly $p^2\mid c,e,g$, so $A^{[2]}=0$, which is a contradiction. Therefore $A^{[2]}$ is generated by $px$ and $py$ as an additive group.
\end{proof}

\begin{prop}
    Let $A$ be a pre-Lie ring with $A^{[2]}\neq 0$ and $A^{[3]}=0$. Suppose the additive group $(A,+)\cong C_{p^2} \times C_{p^2}$ is generated by $x$ and $y$, and that $A$ is generated by one element as a pre-Lie ring. Then the following holds:
    \begin{enumerate}
        \item $A$ is generated by $x$ and $x^2$ as an additive group.
        \item $y = \alpha x+\beta x^2$ for some $\alpha, \beta \in \{1,2,\dots,p^2\}$ such that $p\nmid \beta.$
        \item $x^2\not\in pA$.
        \item $A^{[2]}$ is generated by $x^2$ as an additive group.

    \end{enumerate}
\end{prop}

\begin{proof}
    We can assume $x$ is the generator of $A$ as a pre-Lie ring. Arguing similarly to Proposition \ref{2el} $x\not\in A^{[2]}$. Hence, we can let $y = \alpha x+\beta x^2$ for some $\alpha, \beta \in \{1,2,\dots,p^2\}$. Notice that $p\nmid \beta$ as otherwise $py = p\alpha x$, contradicting the fact that $x$ and $y$ are independent generators of $A$ as an additive group. Now if $x^2\in pA$, then $py = p\alpha x$, again, contradicting the fact that $x$ and $y$ are additive generators of $A$. So $x^2\not\in pA$. We now show $x$ and $x^2$ generate $A$ as an additive group. Suppose $kx = lx^2$ for some $k,l\in\{1,2,\dots,p^2\}$. Further, suppose $p\nmid k$. Then we can write
    \[x=k^{-1}lx^2 = k^{-2}l^2x\c x^2 = 0.\]
    Hence, $p\mid k$. Now, since $x^2\not\in pA$, we have $p\mid l$, so
    \[px = l'px^2 = l'^2px\c x^2 = 0.\]
    Therefore, $x$ and $x^2$ are independent generators of $A$ as an additive group.
\end{proof}

\subsubsection{Pre-Lie rings $A$ with $A^{[3]} \neq 0$ and $A^{[4]} = 0$}

\begin{prop}\label{2genimpliesa3=0}
    Let $A$ be a pre-Lie ring with $A^{[3]}\neq 0$ and $A^{[4]}=0$. Suppose the additive group $(A,+)\cong C_{p^2} \times C_{p^2}$ is generated by $x$ and $y$. If $A$ is generated by two elements as a pre-Lie ring, then $A^{[3]} = 0$.
\end{prop}

\begin{proof}
    We can assume $A$ is generated by $x$ and $y$ as a pre-Lie ring. Then $x\not\in A^{[2]}$, as otherwise $A$ is generated by $1$ element as a pre-Lie ring or $x=0.$ Similarly, $y\not\in A^{[2]}$. 
    Now, $(A^{[2]},+)$ is isomorphic to $C_p, C_p\times C_p, C_{p^2}$ or $C_p\times C_{p^2}$. If $(A^{[2]},+)\cong C_p$, then $A^{[3]}\subseteq pA^{[2]} = 0$. If $(A^{[2]}, +)\cong C_{p^2}$ or $C_p\times C_{p^2}$, then WLOG $A^{[2]}$ contains $x$, which is a contradiction. Suppose $(A^{[2]},+)\cong C_p\times C_p$ is generated by $px$ and $py$. Then $A^{[2]}\cong pA$, so $A^{[3]}=pA^{[2]}=0.$
\end{proof}

Given that in this section we wish to consider pre-Lie rings with $A^{[3]}\neq 0$ and $A^{[4]}=0$, we can assume that $A$ is generated by $y$ as a pre-Lie ring.

\begin{lemma}
Let $A$ be a pre-Lie ring with $A^{[3]}\neq 0$ and $A^{[4]}=0$. Suppose the additive group $(A,+)\cong C_{p^2} \times C_{p^2}$ is generated by $x$ and $y$.
 Then   $(A,+)$ is generated by $y$ and $y^2$ as an abelian group.
\end{lemma}

\begin{proof}
Suppose $y^2\in pA$. Then by Proposition \ref{2genimpliesa3=0} we can write
\[x=\alpha y + \beta y^2 +\gamma y^2\c y +\delta y\c y^2 = \alpha y + pz\]
for some $z\in A$. It follows that $p(x-\alpha y)=0$, which is a contradiction as $x$ and $y$ are independent generators of $(A,+)$.
    Now suppose
    \[ky = ly^2.\]
    Further, suppose $p\nmid k$. Then we have
    \[y = k^{-1}ly^2 = (k^{-1}l)^2y\c y^2 = \dots = 0, \]
    which is a contradiction. Now suppose $p|k$. Then $p|l$ as $y^2\not\in pA = A^{[3]}$. Then we have
    \[py = l'py^2 = l'^2 py\c y^2 = \dots = 0,\]
    for some $l'$, which is a contradiction. Hence $y$ and $y^2$ are independent generators of $(A,+)$.
\end{proof}

Let us fix notation
\[y\c y^2 = ay + by^2,\]
\[y^2\c y = cy + dy^2\]
for $a,b,c,d\in\{1,2,\dots,p^2\}.$

\begin{lemma}
Let $A$ be a pre-Lie ring with $A^{[3]}\neq 0$ and $A^{[4]}=0$. Suppose the additive group $(A,+)\cong C_{p^2} \times C_{p^2}$ is generated by $x$ and $y$. With notation as above, we have
    $p^2|a,c$, $p|b,d$.
\end{lemma}

\begin{proof}
    We have
\[(y\c y^2)\c y = bc y +(a+bd)y^2=0,\]
\[y \c(y\c y^2) = ba y +(a+b^2)y^2=0,\]
\[y \c(y^2\c y) = ad y +(c+bd)y^2=0,\]
\[(y^2\c y)\c y = cd y +(c+d^2)y^2=0.\]

If $p^2\nmid a,$ then $p|b$ which contradicts the second equation. Hence $p^2| a$. It follows that $p^2| bd, b^2$ so $p^2| c$ and $p^2| d^2$. 
\end{proof}

Notice that the pre-Lie relations are satisfied by this structure, so it indeed defines a pre-Lie ring.

  \subsubsection{Pre-Lie rings $A$ with $A^{[4]}\neq 0$}
 In this section we assume $A\neq A^{[2]}\neq A^{[3]}\neq A^{[4]}\neq 0$. Now as $(A^{[2]},+)\leq (A,+)$ and $|A^{[2]}|=p^3$, we have $(A^{[2]},+)\cong C_p\times C_{p^2}$ and we can let $(A^{[2]},+)$ be generated by $x$ and $py$.

\begin{lemma}\label{propxina3}
 Let $A$ be a pre-Lie ring with $A^{[4]}\neq 0$. Suppose the additive group $(A,+)\cong C_{p^2} \times C_{p^2}$ is generated by $x$ and $y$ and $(A^{[2]},+)$ is generated by $x$ and $py$. Then    $x\not \in A^{[3]}$.
\end{lemma}

\begin{proof}

    If $x\in A^{[3]}$, then $(A^{[3]},+)$ is generated by $x$. It follows that $(A^{[4]},+)$ is generated by $px$. Notice that $py\c y \in A^{[3]}$, so $py\c y=kx$ for some $k$ divisible by $p$, so $py\c y\in A^{[4]}$. Also $x\c x, x\c y, y\c x \in A^{[4]}.$
    We have that $A^{[3]} = A^{[2]}\c A+ A\c A^{[2]}$. So elements of $A^{[3]}$ are either of the form $(ax+bpy)\c(cx+dy)$ or $(cx+dy)\c (ax+bpy)$ for some $a,b,c,d\in \Z_{p^2}.$ We have
    \[(ax+bpy)\c(cx+dy) = ac x\c x +adx\c y+bpcy\c x+bpdy\c y,\]
    where all of the terms are in $A^{[4]}$, so $A^{[2]}\c A\subseteq A^{[4]}$. Similarly,  $A\c A^{[2]}\subseteq A^{[4]}$. Hence $A^{[3]}=A^{[4]}$, which implies $A^{[3]}=0$.
\end{proof}

\begin{cor}\label{cora3}
Let $A$ be a pre-Lie ring with $A^{[4]}\neq 0$. Suppose the additive group $(A,+)\cong C_{p^2} \times C_{p^2}$ is generated by $x$ and $y$ and $(A^{[2]},+)$ is generated by $x$ and $py$. Then $(A^{[3]},+)$ is generated by $px$ and $py$, so $A^{[3]} = pA$.
\end{cor}

\begin{lemma}\label{propa4}
Let $A$ be a pre-Lie ring with $A^{[4]}\neq 0$. Suppose the additive group $(A,+)\cong C_{p^2} \times C_{p^2}$ is generated by $x$ and $y$ and $(A^{[2]},+)$ is generated by $x$ and $py$.
  Then  $A^{[4]}=pA^{[2]}$.
\end{lemma}

\begin{proof}
We have $A^{[4]} = A\c A^{[3]}+A^{[2]}\c A^{[2]}+A^{[3]}\c A = pA^{[2]}+A^{[2]}\c A^{[2]}$.
    Notice that $A^{[2]}/A^{[3]} = \langle y^2+A^{[3]}\rangle$, as $x^2, x\c y, y\c x\in A^{[3]}$. Hence, as $x\not\in A^{[3]}$, 
    \[x +A^{[3]}= ky^2+A^{[3]},\]
    where  $p\not | k$. Hence
    \[x\c x +A^{[5]}= k^2 y^2\c y^2+ A^{[5]}.\]
    It follows that $A^{[2]}\c A^{[2]} \subseteq \F_p (y^2\c y^2)+ A^{[3]}\c A^{[2]} + A^{[2]}\c A^{[3]}$.
    For $z\in A^{[2]}$ we have the following relation in A:
    \[(z\c y)\c y-z\c(y\c y) = (y\c z)\c y-y\c(z\c y).\]
    It follows that $A^{[2]}\c A^{[2]}\subseteq A^{[3]}\c A+A\c A^{[3]} = pA^{[2]}$. Hence $A^{[4]} = pA^{[2]}.$
\end{proof}

\begin{cor}
Let $A$ be a pre-Lie ring with $A^{[4]}\neq 0$. Suppose the additive group $(A,+)\cong C_{p^2} \times C_{p^2}$ is generated by $x$ and $y$ and $(A^{[2]},+)$ is generated by $x$ and $py$. Then    $A^{[5]}=0.$
\end{cor}

\begin{lemma}\label{abelian}
Let $A$ be a pre-Lie ring with $A^{[4]}\neq 0$. Suppose the additive group $(A,+)\cong C_{p^2} \times C_{p^2}$ is generated by $x$ and $y$ and $(A^{[2]},+)$ is generated by $x$ and $py$. Then    $(A,+)$ is generated by $y, y^2$ as an abelian group.
\end{lemma}

\begin{proof}
    Suppose
    \[ky = ly^2.\]
    Further, suppose $p\nmid k$. Then we have
    \[y = k^{-1}ly^2 = (k^{-1}l)^2y\c y^2 = \dots = 0, \]
    which is a contradiction. Now suppose $p|k$. Then $p|l$ as $y^2\not\in pA = A^{[3]}$. Then we have
    \[py = l'py^2 = l'^2 py\c y^2 = \dots = 0,\]
    for some $l'$, which is a contradiction. Hence $y$ and $y^2$ are independent generators of $(A,+)$.
\end{proof}

Let us fix notation
\[y\c y^2 = ay +by^2,\]
\[y^2\c y = cy+dy^2\]
for $a,b,c,d\in \{1,2,\dots,p^2\}$.
Then $p|a,b,c,d$ as $y\c y^2, y^2\c y \in pA$. It follows that
\[(y\c y^2)\c y= y\c(y\c y^2)= ay^2,\]
\[(y^2\c y)\c y= y\c(y^2\c y)= cy^2.\]

\begin{lemma}\label{relat}
Let $A$ be a pre-Lie ring with $A^{[4]}\neq 0$. Suppose the additive group $(A,+)\cong C_{p^2} \times C_{p^2}$ is generated by $x$ and $y$ and $(A^{[2]},+)$ is generated by $x$ and $py$. With notation as above, we have
    $y^2\c y^2 = (2c-a) y^2$.
\end{lemma}
\begin{proof}
    Follows from the relation
    \[(y^2\c y)\c y-y^2\c y^2 =(y\c y^2)\c y -  y\c(y^2\c y). \]
\end{proof}

\begin{lemma}
    With notation as above, not both $a$ and $c$ are zero. Further, if $a,c\neq 0$, then $a=\alpha c,$ $b=\alpha d$ for some $p\nmid \alpha$.
\end{lemma}

\begin{proof}
    At least one of $y\c y^2$ and $y^2\c y$ is in $A^{[3]}$. If both are in $A^{[3]}$, then for some $p\nmid\alpha$ we have $y\c y^2=\alpha y^2\c y$.
\end{proof}

Notice that the only pre-Lie relation to be checked is the one in Lemma \ref{relat}, and it is satisfied by definition. Therefore, this structure gives a well-defined pre-Lie ring.

\section{Summary}

In this section we collect the results using the following notation. Suppose the additive group of the pre-Lie ring $A$ is generated by some basis $\{x_1, x_2, \dots , x_n\}$. Then an element of $A$ can be written as $\sum_{i=1}^n a_i x_i$ for some coefficients $a_i\in \Z$.   We introduce notation 
\[[a_1, a_2, \dots, a_n][b_1, b_2, \dots, b_n] = [c_1, c_2, \dots,c_n]\]
to mean that for coefficients $a_i, b_i, c_i \in \Z$ we have 
\[(\sum_{i=1}^n a_i x_i)\c (\sum_{i=1}^n b_i x_i) = \sum_{i=1}^n c_i x_i.\]

\begin{thm}
    Suppose $A$ is a nilpotent pre-Lie ring of cardinality $p^4$. We have the following possibilities:
    \begin{enumerate}
        \item If $(A,+) \cong C_p^4$, $A^{[3]}\neq 0$ and $A^{[4]}=0$, then for $i,j,k,l,m,n,r,s,a,b,c,d,$ $e,f,g,h\in \F_p$ we have
        \[[i,j,k,l][m,n,r,s] = [0,0,ima+inb+jmc+jnd, irf+jrh+kme+kng],\]
        where $a,b,c,d$ are not all zero, $e,f,g,h$ are not all zero, $cg-bh=bg-df$ and $be-cf=ce-ah$.
        \item If $(A,+) \cong C_p^4$ is generated by $2$ elements as a pre-Lie ring, $A^{[2]}\neq 0$ and $A^{[3]}=0$,  then for $i,j,k,l,m,n,r,s,\alpha_{uv}, \beta_{uv}\in \F_p$ we have
        \begin{align*}
            [i,j,k,l][m,n,r,s] = [&0,0,im\alpha_{xx}+in\alpha_{xy}+jm\alpha_{yx}+jn\alpha_{yy}, im\beta_{xx}+\\
            &in\beta_{xy}+jm\beta_{yx}+jn\beta_{yy}],
        \end{align*}
        where at least two of the pairs $(\alpha_{ij}, \beta_{ij})$ do not equal $(0,0)$.
         \item If $(A,+) \cong C_p^4$ is generated by $3$ elements as a pre-Lie ring, $A^{[2]}\neq 0$ and $A^{[3]}=0$,  then for $i,j,k,l,m,n,r,s, \beta_{uv}\in \F_p$ we have
        \[[i,j,k,l][m,n,r,s] = [0,0,0, im\beta_{xx}+in\beta_{xy}+jm\beta_{yx}+jn\beta_{yy}],\]
        where at least one $\beta_{uv}\neq 0$.

        \item If $(A,+) \cong C_p\times C_{p^3}$ and $A^{[3]}\neq 0$, then for $i,k,a,c,e,g\in \{1,2,\dots, p^2\}, b\in \F_p$ we have
        \[[i,j][k,l] = [ika+ilc+jke+jlg, ikb],\]
        where $p^2\mid c,e,g$, $p\mid a$ and $p^3\mid bg$.
        \item If $(A,+) \cong C_p\times C_{p^3}$ is generated by $1$ element as a pre-Lie ring, $A^{[2]}\neq 0$ and $A^{[3]}=0$, then for $a\in \{1,2,\dots, p^3\}$ such that $p^2\mid a$ we have
        \[[i,j][k,l] = [ika,0].\]
        \item If $(A,+) \cong C_p\times C_{p^3}$ is generated by $2$ elements as a pre-Lie ring, $A^{[2]}\neq 0$ and $A^{[3]}=0$, then for $a,c,e,g\in \{1,2,\dots, p^3\}$ such that $p^2\mid a,c,e,g$ we have
        \[[i,j][k,l] = [ika+ilc+jke+jlg, 0].\]

        \item If $(A,+) \cong C_{p^2}\times C_{p^2}$, and $A^{[4]}\neq 0$, then for $i,j,k,l,c,d,e,f,h\in \{1,2,\dots, p^2\}$ we have
        \[[i,j][k,l] = [2ikd-ikf+ilc+jke+jl, ild+jkf+jlh],\]
        where $p\mid c,d,e,f,h$. 
        \item  If $(A,+) \cong C_{p^2}\times C_{p^2}$, $A^{[4]} =0$ and $ A^{[3]}\neq 0$, then for $i,j,k,l,c,e,g,h\in \{1,2,\dots, p^2\}$ we have
        \[[i,j][k,l] = [ilc+jke+jlg,jlh],\]
        where $p\mid c,e,h$ and $p\nmid g$. 
        \item If $(A,+) \cong C_{p^2}\times C_{p^2}$ is generated by $1$ element as a pre-Lie ring, $A^{[2]}\neq 0$ and $A^{[3]}=0$, then for $a,b,\alpha\in \{1,2,\dots, p^2\}$ such that $p\nmid b$ we have
        \[[i,j][k,l] = [ika+il\alpha a+jk\alpha a+jl\alpha^2 a, ikb+il\alpha b+jk\alpha b+jl\alpha^2 b].\]
        \item If $(A,+) \cong C_{p^2}\times C_{p^2}$ is generated by $2$ elements as a pre-Lie ring, $A^{[2]}\neq 0$ and $A^{[3]}=0$, then for $a,b,c,d,e,f,g,h \in \{1,2,\dots, p^2\}$, not all zero, such that $p\nmid a,b,c,d,e,f,g,h$ we have
        \[[i,j][k,l] = [ika+ilc+jke+jlg, ikb+ild+jkf+jlh].\]
        \item $A$ is one of the pre-Lie rings described in \cite{MR4353236}.
    \end{enumerate}
\end{thm}

\cleardoublepage

\bibliographystyle{acm}
\bibliography{bibliography}

\end{document}